\documentclass[a4paper,11pt]{article}
\usepackage{amsmath}
\usepackage{amsfonts}
\usepackage{amssymb}
\usepackage{amsthm}
\usepackage{xspace}
\usepackage[pdftex]{graphicx,color}
\usepackage[pdftex, colorlinks]{hyperref}

\theoremstyle{definition}
\newtheorem{lemma}{Lemma}
\newtheorem*{sublemma}{Sublemma}
\newtheorem{prop}{Proposition}
\newtheorem{thm}{Theorem}
\newtheorem{remark}{Remark}

\newcommand{\ep}{\epsilon}
\renewcommand{\L}{\mathcal{L}}
\newcommand{\C}{\mathcal{C}}
\renewcommand{\t}[1]{\tilde{#1}}

\newcommand{\al}{\alpha}

\newcommand{\ga}{\gamma}
\newcommand{\del}{\delta}
\newcommand{\sig}{\sigma}
\newcommand{\lam}{\lambda}
\newcommand{\rt}{r}
\newcommand{\lt}{l}
\newcommand{\mc}[1]{\mathcal{#1}}
\newcommand{\abs}[1]{\left\vert#1\right\vert}
\newcommand{\acim}{ACIM\xspace}
\newcommand{\acims}{ACIMs\xspace}
\newcommand{\leb}{\text{Leb}}
\newcommand{\bp}{b}
\newcommand{\var}[1]{\text{var}(#1)}
\newcommand{\varsub}[2]{\text{var}_{#1}(#2)}
\newcommand{\lip}[1]{\text{Lip}(#1)}
\newcommand{\comment}[1]{#1}
\newcommand{\cly}{C_{\text{LY}}}
\newcommand{\cdis}{C_{\text{dis}}}

\title{Approximating invariant densities of metastable systems}
\author{Cecilia Gonz\'alez Tokman\thanks{Department of Mathematics, University of Maryland, College Park, MD 20742.  Email:  cecilia@math.umd.edu.  This research is partially supported by CONACyT, M\'exico.}
\and Brian R. Hunt\thanks{Department of Mathematics and Institute for Physical Sciences and Technology, University of Maryland, College Park, MD 20742, USA. Email: bhunt@umd.edu.}
\and Paul Wright\thanks{Department of Mathematics, University of Maryland, College Park, MD 20742.
Email:  paulrite@math.umd.edu.  This research is partially supported
by an NSF Mathematical Sciences Postdoctoral Research Fellowship.}}

\begin{document}
\maketitle
\begin{abstract}
We consider a piecewise smooth expanding map of the interval possessing two invariant subsets of positive Lebesgue measure and exactly two ergodic absolutely continuous invariant probability measures (\acims). When this system is perturbed slightly to make the invariant sets merge, we describe how the unique \acim of the perturbed map can be approximated by a convex combination of the two initial ergodic \acims.
\end{abstract}


\section{Introduction}

Metastable systems are studied in relation with phenomena ranging from molecular \cite{MeerbachEtAl} to oceanic \cite{FroylandPadbergEtAl} dynamics.
Typical trajectories of these systems remain in one of its almost invariant (metastable or quasi-stationary) components for a relatively long period of time, but eventually switch to a different component and repeat this behavior.
Quantitative aspects of these phenomena have been studied through eigenvalue and eigenvector approximation techniques for Markov models \cite{MeerbachSchutteFischer, FroylandPadberg}.  Here, we are concerned with rigorous approximation results for eigenvectors--in particular those that correspond to stationary measures of the dynamics--in a more general (non-Markov) setting.

Broadly, our setting concerns the approximation of absolutely continuous invariant probability measures (\acims) for certain hyperbolic maps with metastable states. These systems arise from perturbing an initial system $T_0$ with two disjoint invariant sets $I_\lt$, $I_\rt $ of positive Lebesgue measure. The initial map has two mutually singular ergodic \acims, $\mu_\lt$ and $\mu_\rt$.  When $T_0$ is perturbed in such a way that $I_\lt$ and $I_\rt $ lose their invariance and the perturbed map $T_\ep$ has only one \acim $\mu_{\ep}$, we are interested in approximating $\mu_{\ep}$ using $\mu_\lt$ and $\mu_\rt$.  Specifically, the systems we consider are piecewise $C^2$ expanding maps of an interval; see Figure~\ref{figure:2xmod1}.

Our results can be understood in the context of dynamical systems with holes as follows.
As the invariance of the two initially invariant sets is destroyed by the perturbation, we think of the small set of points $I_\lt\cap T_\ep^{-1}I_\rt$ that switch from $I_\lt$ to $I_\rt$, and likewise the set $I_\rt\cap T_\ep^{-1}I_\lt$, as being holes in the initially invariant sets. From this point of view we expect to be able to approximate $\mu_\ep $, for small $\ep$,  by a convex combination $\al\mu_{\lt}+(1-\al)\mu_\rt$ of the two initially invariant measures, with the ratio $\al/(1-\al)$ depending on the relative sizes of the holes.

\comment{
\begin{figure}[htbp]
\begin{center}
\label{figure:2xmod1}
\resizebox{5cm}{!}{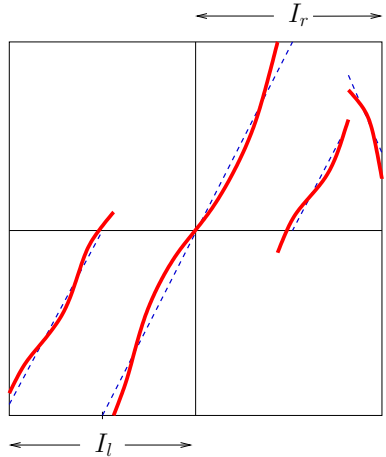}
\caption{Dashed: initial system. Thick: metastable system.}
\end{center}
\end{figure}
}

Before discussing our results, we present two illustrative examples.  We begin with a simple random system.
Consider the family of Markov chains in two states $\lt$ and $\rt$, with transition matrices
\[
Q_\ep=\begin{pmatrix}
1-\ep_{\lt \rightarrow \rt} &  \ep_{\lt \rightarrow \rt}\\
\ep_{\rt \rightarrow \lt} & 1-\ep_{\rt \rightarrow \lt}
\end{pmatrix},
\]
where $\ep=(\ep_{\lt \rightarrow \rt},\ep_{\rt \rightarrow \lt}) $.   We are interested in the behavior when $\ep\approx 0$. When $\ep=0$, the two sets $I_\lt=\{\lt\}$ and $I_\rt=\{\rt\}$ are invariant, giving rise to the two ergodic stationary probability measures $\mu_\lt=\delta_{\lt}$ and $\mu_\rt=\delta_\rt$.  When $\ep_{\lt \rightarrow \rt}> 0$, there is a unique stationary probability measure
\[
\mu_\ep=\al_\ep\mu_{\lt}+(1-\al_\ep)\mu_\rt, \text{ where } \frac{\al_\ep}{1-\al_\ep}=\frac{\ep_{\rt \rightarrow \lt}}{\ep_{\lt \rightarrow \rt}}.
\]
Observe that the ratio of the weights $\al_\ep/(1-\al_\ep)$, i.e.~$\mu_\ep(I_\lt)/\mu_\ep(I_\rt)$, is equal to the inverse ratio of the sizes of the holes, $\ep_{\rt \rightarrow \lt}/\ep_{\lt \rightarrow \rt}$.

Next, we consider two billiard tables $\mathcal{D}_\lt, \mathcal{D}_\rt$ in the plane, as indicated in Figure~\ref{figure:billiards}.  For $*\in\{\lt,\rt\} $, let $T_*:I_*\circlearrowleft $ be the corresponding billiard map, i.e.~the Poincar\'{e} map for the first return of the billiard flow to $\partial \mathcal{D}_* $.  We use $\abs{\partial \mathcal{D}_*} $ to denote the perimeter of $\mathcal{D}_*$.  A general reference for hyperbolic billiards is \cite{ChernovMarkarian}, where one can find the background for the assertions below.  We use the usual coordinates $(s,\varphi) $ on $I_* $, where $s$ is arc length on $\partial \mathcal{D}_* $, and $\varphi\in[-\pi/2,+\pi/2]$ is the angle between the outgoing velocity vector and the inward pointing normal vector to $\partial \mathcal{D}_* $.  Then it is well known that $T_*$ leaves (normalized) Liouville measure $\mu_* $ invariant, where $\mu_* $ has the density $\phi_* := d\mu_* /ds\, d\varphi = [2\abs{\partial \mathcal{D}_*}]^{-1} \cos \varphi $.  Next, for $\ep>0 $, let $h_\ep $ be a subsegment of  $\partial\mathcal{D}_\lt\cap \partial\mathcal{D}_\rt$ of length $\ep$, and let $\mathcal{D}_\ep$ be the billiard table resulting after $h_\ep $ is removed.    The corresponding density for the invariant Liouville measure of the billiard map is $\phi_\ep=[2(\abs{\partial \mathcal{D}_\lt}+\abs{\partial\mathcal{D}_\rt}-2\ep)]^{-1} \cos \varphi $.   Thus as $\ep\rightarrow 0$,
\[
    \phi_\ep\rightarrow \al\phi_{\lt}+(1-\al)\phi_\rt, \text{ where } \frac{\al}{1-\al}=\frac{\abs{\partial \mathcal{D}_\lt}}{\abs{\partial \mathcal{D}_\rt}},
\]
provided some care is taken to define all of the density functions involved on the same space.
Note that if we define the holes $H_{*,\ep}:=T_*^{-1} (h_\ep \times [-\pi/2,+\pi/2])$, then we can rewrite
$
    \al/(1-\al)=\mu_\rt (H_{\rt,\ep})/\mu_\lt (H_{\lt,\ep}),
$
so that again the ratio of the weights equals the inverse ratio of the sizes of the holes.
This example is most meaningful when $T_\lt$, $T_\rt$, and  $T_\ep$ are all ergodic, which is the case for the tables in Figure~\ref{figure:billiards}; see \S8.15 in \cite{ChernovMarkarian}.

\comment{
\begin{figure}[htbp]
\begin{center}
\resizebox{4cm}{!}{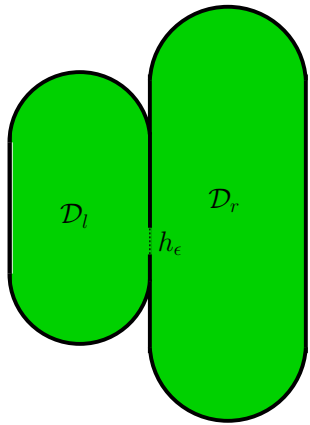}
\caption{Two ergodic billiard tables connected by a hole.}
\label{figure:billiards}
\end{center}
\end{figure}
}

In our main result, Theorem~\ref{thm1}, we show a corresponding result in the deterministic setting of piecewise $C^2$ expanding maps, under fairly general conditions described in \S\ref{S:model}.  We show that as $\ep\rightarrow 0$ the invariant density $\phi_\ep$ of $T_\ep$ converges in $L^1$ to a convex combination of the ergodic invariant densities of $T_0$, with the ratio of the weights given by the limiting inverse ratio of the sizes of the holes.  We emphasize that our results do not require any of the piecewise expanding maps involved to have a Markov partition.

The density $\phi_\ep$ corresponds to an eigenvector with eigenvalue 1 for the Perron-Frobenius operator acting on a suitable space of functions.  Our assumptions imply that for $\ep>0$, the  operator has 1 as a simple eigenvalue, and also another real simple eigenvalue slightly less than 1. In Theorem~\ref{thm2}, we characterize the eigenvectors of this lesser eigenvalue by showing that asymptotically they lie on the line spanned by $d\mu_\lt/dx-d\mu_\rt/dx $.

Unlike the two examples above, in the setting of piecewise $C^2$ expanding maps we have no explicit formulas for the invariant densities; even their existence is nontrivial.  Our methods rely on the fact that the densities of the \acims for $T_{\ep}$ are of bounded variation \cite{LasotaYorke}. Hence, they can be decomposed into regular and singular (or saltus) parts, as in \cite{Baladi}.
The key technical portions of our proofs include estimating and exploiting the locations and sizes of the jumps at the discontinuities of the invariant densities, which occur on the forward trajectories of the critical points of $T_\ep $.

This work is related to other recent work involving metastable systems and piecewise expanding maps.  Recently, \cite{KellerLiverani} studied metastable systems arising from piecewise smooth uniformly expanding maps with two invariant intervals. They perturbed such an initial map by a family of Markov operators close to the identity to produce a family of metastable systems for $\ep>0 $.  The associated Perron-Frobenius  operators acting on a suitable space of functions have $1$ as a simple eigenvalue and another simple eigenvalue $\rho_\ep<1 $.  As $\ep\rightarrow 0$, $\rho_\ep\rightarrow 1 $, and the authors rigorously computed the derivative $\lim_{\ep\rightarrow 0^+}(1-\rho_\ep)/\ep$. This provides information on the stationary exchange rate between the metastable states. Their work may be used to show a corresponding result in our setting.

Our work is also related to current and ongoing investigations on linear response.
These problems have the feature that, as Ruelle \cite{Ruelle} puts it, it is possible to formulate conjectures based on intuition or formal calculations, but the proofs often involve overcoming intricate technicalities.  In our setting, we know that $\mu_\ep(I_\lt)\rightarrow\al$ as $\ep\rightarrow 0$.  A pertinent open problem would be to try and characterize the higher-order terms $\mathcal{R}(\ep):=\mu_\ep(I_\lt)-\al $.  We do not expect $\mathcal{R}(\ep)$ to be differentiable at $\ep=0$ in general.  As shown in~\cite{BaladiSmania}, linear response fails precisely when the perturbations $T_\ep $ are transverse to the topological class of $T_0 $, at least for certain piecewise expanding unimodal maps $T_0 $ that are topologically mixing.  Results of~\cite{Keller82} show that in that setting the unique \acim $\phi_\ep $ of $T_\ep$ satisfies $\abs{\phi_\ep-\phi_0}_{L^1}=O(\ep\log\ep) $, where $\phi_0 $ is the unique \acim of $T_0$.  \cite{Baladi} gives examples where this estimate is optimal.

Another problem for further research is extending our results to higher dimensional piecewise hyperbolic maps. While we use techniques specific to one-dimensional maps, we are optimistic that the main elements of our proof, found in \S\ref{S:main_proof}, can be generalized.

\section{Statement of results}\label{S:model}
In this section, we define a class of dynamical systems with two
nearly invariant (metastable) subsets. They are
perturbations of a one-dimensional piecewise smooth expanding map
with exactly two invariant subintervals $I_\lt$ and $I_\rt$ of positive Lebesgue measure.
On each of
these intervals, the unperturbed system has a unique \acim.
The perturbations break this invariance by introducing what we consider to be holes
in the intervals; the hole(s) in $I_\lt$ map to $I_\rt$ and vice versa.  Each perturbed
system will have only one \acim, and we will determine an asymptotic
formula for its density in terms of the invariant densities of the
unperturbed system.

Let $I=[0,1].$ In this paper, a map $T:I\circlearrowleft$ is called
a piecewise $C^2 $ map with $\C=\{ 0=c_0 < c_1<  \cdots < c_d=1\}$
as a critical set if for each $i$, $T\vert_{(c_i,c_{i+1})}$ extends
to a $ C^2$ function on a neighborhood of $[c_i,c_{i+1}]$.  We call $T$ uniformly expanding if its minimum expansion, $\inf_{x\in I\setminus \C_0}|T_0'(x)|$, is greater than 1.  As is customary for piecewise smooth maps, we consider $T$ to be bi-valued at points $c_i\in\C$ where it is discontinuous. In such cases we let
$T(c_i)$ be both values obtained as $x$ approaches $c_i$ from either side, and $T(c_{i\pm})$ the corresponding right and left limits.
If $a,b\in\C$, $T\vert_{[a,b]}$ will be used to specifically denote
the restriction of $T$ with $T\vert_{[a,b]}(a)=T(a_+)$ and $T\vert_{[a,b]}(b)=T(b_-)$.

We use $\leb $ to denote normalized Lebesgue measure on $I$ and $L^1 $ to denote the space of
Lebesgue integrable functions on $I$, with norm $\abs{f}_{L^1} =\int_I \abs{f(x)}\,dx $.
Also, for $f:I\rightarrow \mathbb{C}$, we let $\abs{f}_\infty$ be the supremum of $f$ over $I$ and $\var{f} $ be the total variation of $f$ over $I$; that is,
\[
\var{f}= \sup\{ \sum_{i=1}^n|f(x_{i})-f(x_{i-1})|: n\geq 1, 0\leq x_0<x_1<\dots<x_n\leq 1 \}.
\]

For clarity of presentation, we do not state our results under the
broadest possible assumptions.  However, see \S\ref{S:gen} for a
number of relaxations of the hypotheses below.

\subsection{The initial system and its perturbations}\label{S:T_0}
We assume that the unperturbed system is a piecewise $ C^2$ uniformly expanding map
$T_0:I\circlearrowleft$ with $\C_0=\{0= c_{0,0}< c_{1,0}<  \cdots <
c_{d,0}=1\}$ as a critical set.  There is a boundary point $\bp\in(0,1)$ such that $I_\lt :=
[0,\bp]$ and $I_\rt := [\bp,1]$ are invariant under $ T_0$,
i.e.~for $*\in\{\lt,\rt\}$, $T_0\vert_{I_*}(I_*)\subset I_*$.
The existence of an \acim~of bounded variation for $T_0\vert_{I_*}$ is guaranteed by \cite{LasotaYorke}.  We assume in addition:

\begin{enumerate}

\item[(I1)] \emph{Unique \acims on the initially invariant set.}\\
$T_0\vert_{I_*}$ has only one \acim~$\mu_*$, whose density is denoted by
$\phi_* := d\mu_*/dx$.
\end{enumerate}
The uniqueness of such an
\acim can be guaranteed by transitivity or by additional conditions described in \cite{LiYorke}.
From (I1), it follows that all \acims~of $T_0$ are
convex combinations of the ergodic ones, $\mu_\lt$ and $\mu_\rt$.

We define the points in $H_0:=T_0^{-1}\{\bp\} \setminus \{\bp\}$ to be \emph{infinitesimal
holes}. These are all points that map to the boundary point $\bp$, except possibly $\bp$ itself.  Our reasons for excluding $b $ from the set of infinitesimal holes will be explained in \S\ref{S:gen}.
An immediate consequence of this definition is that $H_0\subset \C_0$.

\begin{enumerate}
\item[(I2)] \emph{No return of the critical set to the infinitesimal holes}. \\
For every $k>0$, $(T_0^k \C_0)\cap H_0=\emptyset$.
\end{enumerate}
As we will see in \S\ref{S:densitiesproof}, this implies that $\phi_*$ is
continuous at each of the infinitesimal holes in $I_*$.

\begin{enumerate}
\item[(I3)] \emph{Positive \acims~at infinitesimal holes.\\}
$\phi_\lt$  is positive at each of the points in $H_0\cap I_\lt $, and $\phi_\rt$  is positive at each of the points in $H_0\cap I_\rt $.
\end{enumerate}
For example, this will be the case if $T_0\vert_{I_\lt}$ and $T_0\vert_{I_\rt}$ are weakly
covering,\footnote{A piecewise expanding map $T:I\circlearrowleft$ with $\C=\{ 0=c_0 < c_1<  \cdots < c_d=1\}$ as a critical set
is weakly covering if there is some $N$ such that for every $i$, $\cup_{k=0}^{N}T^k([c_i,c_{i+1}])=I$} see~\cite{Liverani}.

\begin{enumerate}
\item[(I4)] \emph{Restriction on periodic critical points.\\}
Either
\begin{enumerate}
\item[(I4a)] $\inf_{x\in I\setminus
\C_0}|T_0'(x)|>2$, or
\item[(I4b)] $T_0 $ has no periodic critical points, except possibly that $0 $ or $1 $ may be fixed points.
\end{enumerate}
\end{enumerate}
Because $T_0 $ may be bi-valued at points in $\C_0$, a critical point $c_{i,0} $ is considered periodic if there exists $n>0 $ such that $c_{i,0}\in T_0^n(c_{i,0})$.
Condition (I4) is necessary in order to ensure that the perturbed systems defined below satisfy uniform Lasota-Yorke estimates.
Since we cannot exclude the possibility of the forward orbit of a critical point containing other critical points, these uniform estimates do not follow directly from the original paper \cite{LasotaYorke}, but rather from later works, see \S\ref{S:densitiesproof}.

For what follows, we consider $C^2$-small perturbations
$T_{\epsilon}:I\circlearrowleft$ of $T_0$ for $\ep>0$.  This means that a critical set for $T_{\ep}$ may be chosen as
 $\C_\ep=\{0=c_{0,\ep}< c_{1,\ep}<  \cdots <
c_{d,\ep}=1\}$, where for each $i $, $\ep\mapsto c_{i,\ep}$ is a $C^2 $ function for $\ep\geq 0 $.  Furthermore, there exists  $\del>0 $ such that for all sufficiently small $\ep$, there exists a $C^2 $ extension $\hat T_{i,\ep}:[c_{i,0}-\del,c_{i+1,0}+\del]\rightarrow \mathbb{R}$ of $T_{\ep}\vert_{[c_{i,\ep},c_{i+1,\ep}]}$, and $\hat T_{i,\ep} \rightarrow \hat T_{i,0}$ in the $C^2$ topology.
We also assume:
\begin{enumerate}

\item[(P1)] \emph{Unique \acim.}\\
 For $\ep>0$, $T_{\epsilon}$ has only one \acim~$\mu_\ep$, with density
 $\phi_\ep := d\mu_\ep/dx$.

\item[(P2)] \emph{Boundary condition.} \\
The boundary point does not move, and no holes are created near the boundary; precisely,
\begin{itemize}
  \item[(P2a)] If $b\notin \C_0$, then necessarily $T_0(\bp)=\bp$. We assume further that for all $\ep>0 $, $T_\ep(\bp)=\bp$.
  \item[(P2b)] If $b\in \C_0$, we assume that $T_{0}(\bp_-)< \bp<T_{0}(\bp_+)$, and also that $b\in \C_\ep$ for all $\ep $.
\end{itemize}
\end{enumerate}

If the boundary point does move under the perturbation, condition (P2) often can be satisfied by performing a small change of coordinates; see \S\ref{S:gen}.

\subsection{Main results}\label{S:results}

The central question of this study is, for small $\ep$, how
can we asymptotically approximate $\mu_{\ep}$ by a convex combination of
$\mu_\lt$ and $\mu_\rt$?  To that end, let $H_{\lt,\ep}:=I_\lt\cap
T_{\ep}^{-1}(I_\rt)$ and $H_{\rt,\ep}:=I_\rt\cap
T_{\ep}^{-1}(I_\lt)$.  We refer to these sets as \textit{holes}.
Once a $T_\ep$-orbit enters a hole, it leaves one of the invariant
sets for $T_0$ and continues in the other. As $\ep\rightarrow 0$,
the holes converge  (in the Hausdorff metric) to the infinitesimal holes from which they arise.

Condition (P1) ensures that for $\ep>0$, at least one of the holes
has positive Lebesgue measure.  In view of (I3), without loss of generality, we
suppose that $\mu_{\lt}(H_{\lt,\ep})>0$ and define
\[
    l.h.r.= \lim_{\ep\rightarrow 0}
\frac{\mu_\rt(H_{\rt,\ep})}{\mu_\lt(H_{\lt,\ep})},
\]
if the limit exists.  ($l.h.r. $ stands for \emph{limiting hole ratio}.)

\begin{thm}[\textbf{Approximation of the invariant density}]\label{thm1}
Consider the family of perturbations $T_\ep$ of $T_0$ under the
assumptions stated in \S\ref{S:T_0}.  Suppose that $l.h.r.$ above exists.  Then
as $\ep\rightarrow 0$,
\[
    \phi_\ep\overset{L^1}{\longrightarrow} \al
    \phi_\lt+(1-\al)\phi_\rt,
    \quad\text{ where \:\;}\frac{\al}{1-\al}=l.h.r..
\]
\end{thm}

We allow for $l.h.r.=+\infty$, in which case $\al=1$.  Several straightforward generalizations of the above result are
discussed in \S\ref{S:gen}.

\begin{remark}
The limit $l.h.r.$ will always exist as long as the perturbations open up holes $H_{\lt,\ep} $ whose size is truly first order in $\ep$:
For simplicity, suppose that there are only two infinitesimal holes, $h_\lt\in I_\lt$ and $h_\rt\in I_\rt$.
Then we can always write $H_{*,\ep}=(h_*-a_*\ep+o(\ep),h_*+b_*\ep+o(\ep))$ for
$*\in\{\lt,\rt\}$, and if $a_\lt+b_\lt>0$, then
\[
    l.h.r.=
    \frac{\phi_\rt(h_\rt)(a_\rt+b_\rt)}{\phi_\lt(h_\lt)(a_\lt+b_\lt)}.
\]
For example, this will be the case if
$T_\ep=T_0+\ep g +o(\ep)$ for some smooth function $g$ with $g(h_l)>0$.
\end{remark}

\begin{remark}
An alternative definition of $l.h.r.$ is as a limit of a
ratio of escape rates: For $*\in\{\lt,\rt\}$, we can consider a
dynamical system with a hole, where orbits stop upon entering
the hole, by using the unperturbed map $T_0\vert_{I_*}$ with
$H_{*,\ep}$ as the hole. See \cite{DemersYoung} for an exposition of
such systems.  Let $R_{*,\ep}$ be the exponential escape rate of Lebesgue measure and suppose that there is only one infinitesimal hole in each initially invariant interval. Then as $\ep\rightarrow 0$, $
\mu_*(H_{*,\ep})/ R_{*,\ep}\rightarrow 1$. \cite{BunimovichYurchenko, KellerLiverani}
\end{remark}

Next, let $\L_\ep$ be the Perron-Frobenius  operator associated with
$T_\ep$ acting on the Banach space $BV=\{f:I\rightarrow \mathbb{C}:\var
{f}<\infty \}$\footnote{Technically, two functions of bounded variation are considered equivalent if they differ on at most a countable set.  Here and elsewhere, we generally ignore such distinctions.} with the variation norm, and let $\sig(\L_\ep) $ denote the spectrum of $\L_\ep $.  It follows from e.g. \cite[Thm. 8.3(b)]{Keller89} that $\L_0$ has one as an isolated eigenvalue of multiplicity two.
Furthermore in \cite{KellerLiverani99} the authors show that
for fixed small $\del>0$ and for every $\ep>0$ small enough, $\sig(\L_\ep)\cap B_\del (1)$ consists of exactly two eigenvalues, 1 and $\rho_\ep<1$, each of multiplicity 1.  As $\ep\rightarrow 0$,
$\rho_\ep\rightarrow 1$ and the total spectral projection of $\L_\ep$ associated with $\sig(\L_\ep)\cap B_\del (1)$ converges (at a given rate in an appropriate norm) to the total spectral projection of $\L_0$ associated with $\sig(\L_0)\cap B_\del (1)$.  Note that in \S\ref{S:densitiesproof} we will show that the assumptions of \cite{KellerLiverani99} are satisfied in our context.

\begin{thm}[\textbf{Characterization of the eigenspace corresponding to the lesser eigenvalue}]\label{thm2}

For each $\ep>0$ small enough, there is a unique real-valued function $\psi_\ep\in BV$
satisfying $\L_\ep\psi_\ep=\rho_\ep\psi_\ep$,
$\abs{\psi_\ep}_{L^1}=1$, and $\int_{I_\lt} \psi_\ep dx>0$. As
$\ep\rightarrow 0$,
\[
    \psi_\ep\overset{L^1}{\longrightarrow} \frac{1}{2}
    \phi_\lt-\frac{1}{2}\phi_\rt.
\]
\end{thm}

\begin{remark}
Suppose $\mu_\lt$ and $\mu_\rt$ are both mixing for $T_0$.
Given a typical initial density $f\in BV$ (i.e. one with nonzero coefficient of $\psi_\ep$ when expressed as a linear combination of eigenvectors), as $\L_{\ep}^n f \rightarrow \phi_{\ep}$, the deviation  $\L_{\ep}^n f - \phi_{\ep}$ becomes approximately proportional to $\psi_\ep$ for $n$ large. In this case, Theorem~\ref{thm2} implies roughly that for $\ep$ small, $\L_{\ep}^n f$ becomes close to a linear combination of $\phi_l$ and $\phi_r$ more quickly than it comes close to the specific linear combination $ \al\phi_\lt+(1-\al)\phi_\rt$.
\end{remark}

\subsection{Examples}\label{S:ex}

The three piecewise linear maps shown in Figure~\ref{figure:examples} satisfy assumptions (I1)-(I4) from \S\ref{S:T_0}.
In all three cases, normalized Lebesgue measure restricted to the left or right intervals is the unique \acim of the corresponding restricted system.

\comment{
\begin{figure}[htbp]
\begin{center}
\resizebox{9cm}{!}{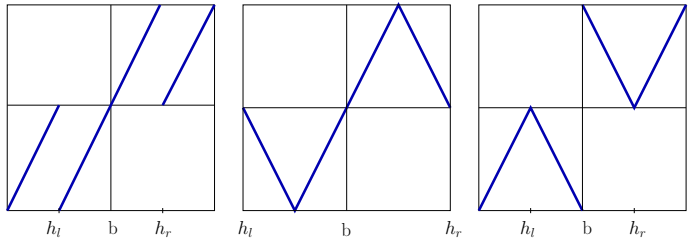}
\caption{Piecewise linear maps giving rise to metastable systems.}
\label{figure:examples}
\end{center}
\end{figure}
}

Adding a small $C^2$ perturbation $g:I\times [0,\ep_0)\rightarrow I$ such that $g(\cdot,0)\equiv0$ and for $\ep\neq 0$, $g(\bp,\ep)=0$, $g(h_\lt,\ep)>0$ and $g(h_\rt,\ep)<0$ gives a one-parameter family of perturbations $T_{\ep}:=T_0+g(\cdot,\ep)$ satisfying assumptions (P1) and (P2).

If $\lim_{\ep\rightarrow0}\frac{ \leb(H_{\rt,\ep}) }{ \leb(H_{\lt, \ep}) }=l.h.r.$, by Theorem~\ref{thm1}, the invariant densities
$\phi_{\ep}$ associated to $T_{\ep}$ satisfy
\[
    \phi_\ep\overset{L^1}{\longrightarrow} \al \leb|_{I_\lt}+(1-\al)\leb|_{I_\rt},
    \text{ where \:\;}\frac{\al}{1-\al}=l.h.r..
\]
The possibility $l.h.r.=\infty$ is allowed, and in this case,
\[
   \phi_\ep\overset{L^1}{\longrightarrow}  \leb|_{I_\lt}.
\]

Other initial maps $T_0$ for which Theorems~\ref{thm1} and \ref{thm2} are applicable are shown in  Figure~\ref{figure:moreExamples}.
\comment{
\begin{figure}[htbp]
\begin{center}
\resizebox{7cm}{!}{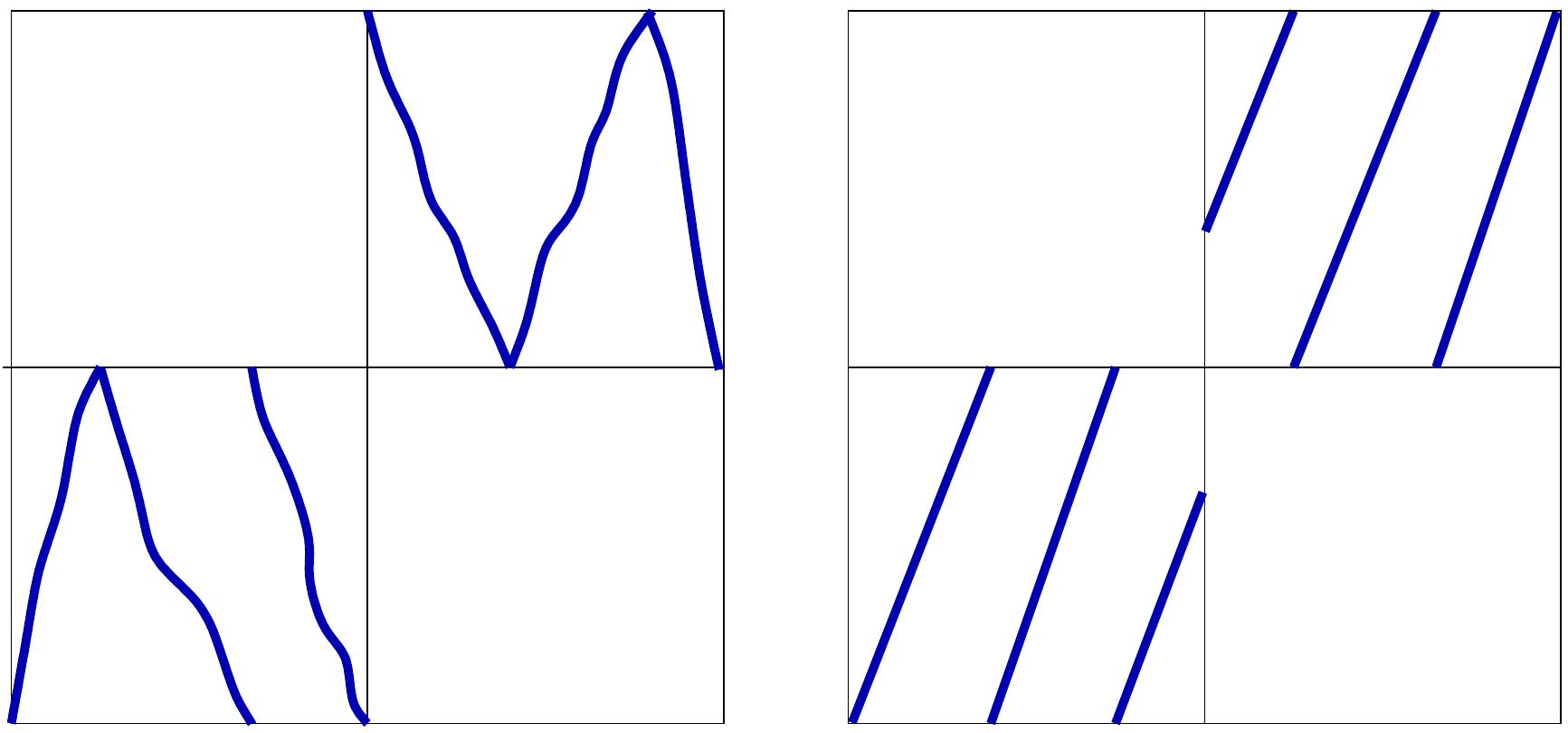}
\caption{Examples of initial maps $T_0 $ which give rise to metastable
systems for which our results hold.}
\label{figure:moreExamples}
\end{center}
\end{figure}
}

\subsection{Generalizations}\label{S:gen}

Our results extend, with essentially the same proofs, to yield the following straightforward generalizations.

\subsubsection*{Multiple invariant sets.}
We can also allow $T_0$ to have $m\geq 2$  invariant sets, provided it has a unique \acim  on each of them. The invariant sets may be intervals or a union of intervals. See Figure~\ref{figure:othermaps}. In this case, the unique invariant density $\phi_\ep$ of $T_\ep$ converges to a convex combination of the initial ergodic invariant densities as $\ep\rightarrow 0$. The coefficients may be determined from $m-1 $ linear equations involving limits of the quotients of measures of appropriate holes.  If $m>2 $, the analogue of Theorem~\ref{thm2} says only that the eigenfunctions for $T_\ep$ whose eigenvalues approach 1 limit on the space of eigenfunctions for $T_0$ with integral 0.

\comment{
\begin{figure}[htbp]
\begin{center}
\resizebox{7cm}{!}{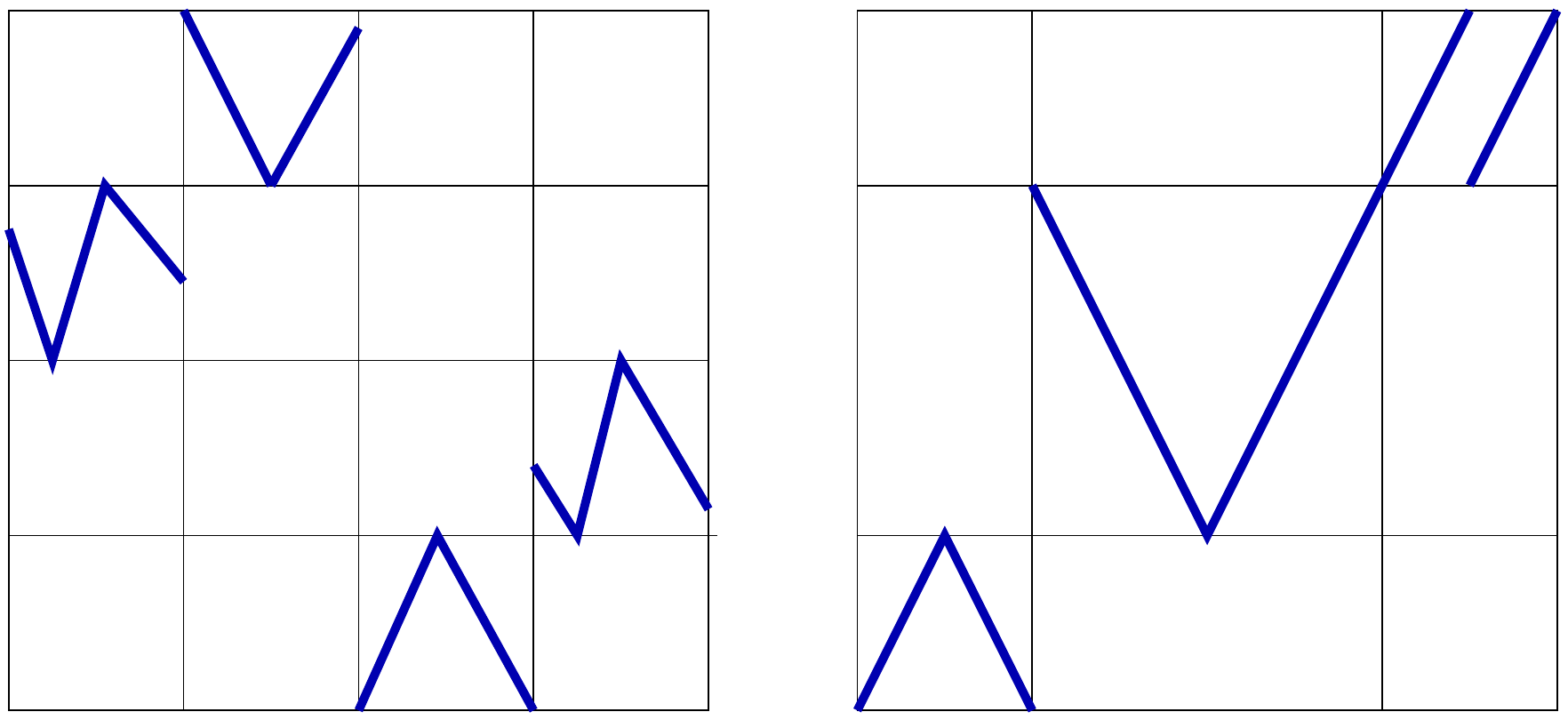}
\caption{Initial maps $T_0 $ which give rise to metastable systems for which our results can be generalized.}
\label{figure:othermaps}
\end{center}
\end{figure}
}

\subsubsection*{Boundary condition.}
The restriction that the boundary point does not move when $T_0 $ is perturbed is inessential; when it is relaxed, it simply means that the metastable states for $T_\ep $ are slight perturbations of the initial invariant sets.
In this case, a smooth change of coordinates restores the hypothesis (P2).
For example, when $\bp\notin \C_0$ assumption (P2a) is actually superfluous, although the definitions in the statement of Theorem~\ref{thm1} must be modified slightly.  As remarked earlier, necessarily $T_0(\bp)=\bp$.  Furthermore, the graph of $T_0 $ intersects the diagonal transversely at this point.  Thus for all small $\ep>0$, there is a unique point $\bp_\ep$ near $\bp $ satisfying $T_\ep(\bp_\ep)=\bp_\ep$.  Then the quasi-invariant sets for $T_\ep $ are $I_{\lt,\ep}:= [0,\bp_\ep] $ and $I_{\rt,\ep}:= [\bp_\ep,1] $, and the corresponding holes are defined by $H_{\lt,\ep}:=I_{\lt,\ep}\cap
T_{\ep}^{-1}(I_{\rt,\ep})$ and $H_{\rt,\ep}:=I_{\rt,\ep}\cap
T_{\ep}^{-1}(I_{\lt,\ep})$.  Aside from these minor modifications, the statements and proofs of our main results remain the same.

When $\bp\in \C_0$, (P2b) can be relaxed by no longer requiring that $\bp\in\C_\ep $ for all $\ep$.  In this case, when $\ep>0$, $\C_\ep$ contains a point $\bp_\ep $ that converges to $\bp $ as $\ep\rightarrow 0$, and the quasi-invariant sets and holes must be redefined as above.  However, it is still essential to assume that no holes are created near the boundary, which we enforce with the assumption that $T_{0}(\bp_-)< \bp<T_{0}(\bp_+)$.  For example, if $T_\ep(x)=[(3x \mod 1/2) +3\ep]\cdot 1_{x<1/2} +[(-3x \mod 1/2) +1/2-\ep]\cdot 1_{x>1/2}$, then all of our assumptions aside from (P2) hold, with $\bp=1/2 $, $\mu_*=\leb|_{I_*}$, and $l.h.r.=1/3 $.  However, as $\ep\rightarrow 0$, $\phi_\ep\overset{L^1}{\longrightarrow}\phi_\rt $.  The difficulty is that orbits ejected from $I_\rt$ by $T_\ep $ immediately return to $I_\rt $.

\subsubsection*{Multiple limiting densities.}
When the limit $l.h.r.$ in \S\ref{S:results} does not exist, we let
\[
\underline{l.h.r.}=\liminf_{\ep\rightarrow 0}\frac{\mu_\rt(H_{\rt,\ep})}{\mu_\lt(H_{\lt,\ep})},\quad
\overline{l.h.r.}=\limsup_{\ep\rightarrow 0} \frac{\mu_\rt(H_{\rt,\ep})}{\mu_\lt(H_{\lt,\ep})}.
\]
Since the function
$\frac{\mu_\rt(H_{\rt,\ep})}{\mu_\lt(H_{\lt,\ep})}$ is continuous in $\ep>0$,
our arguments show that the set of limit points for $\phi_\ep$ as $\ep\rightarrow0$ is precisely
\[
\left\{ \t{\alpha}\phi_\lt+(1-\t{\alpha})\phi_\rt: \frac{\t{\alpha}}{1-\t{\alpha}}\in [\underline{l.h.r.},\overline{l.h.r.}]\right\}.
\]

\section{Proofs of the main theorems}\label{S:mainthm}

In this section, we briefly state the main properties of the invariant densities that will be needed before presenting the proofs of Theorems~\ref{thm1} and \ref{thm2}.  For notational convenience, we will assume that there are only two infinitesimal holes, $h_\lt\in I_\lt$ and $h_\rt\in I_\rt$; the proof without this restriction is essentially unchanged.

\subsection{Properties of the invariant densities} \label{S:propdensities}

Here we record some of the relevant characteristics of the density functions $\phi_{\ep},\phi_\lt,\phi_\rt$.  First, if $f\in BV $, we can -- and will -- choose a representative of $f$ with only regular discontinuities, i.e.~for each $x$, $f(x)=(\lim_{y\rightarrow x^-}f(y)+\lim_{y\rightarrow x^ +}f(y))/2$.  Then, following~\cite{Baladi}, we can uniquely decompose $f=f^{reg}+f^{sal}$ into the sum of a regular and a singular (or saltus) part.  Here $f^{reg}$ is continuous with $\var{f^{reg}}\leq \var{f} $, and $f^{sal}$ is the sum of at most countably many step functions. We write $f^{sal}=\sum_{u\in\mc{S}} s_u H_{u},$ where $\mc{S} $ is the discontinuity set of $f $, $s_u $ is the \emph{jump} of $f $ at $u$, and $H_u(x)=-1$ if $x<u$, $-\frac{1}{2}$ if  $x=u$ and 0 if $x>u$. This representation imposes the boundary condition $f^{sal}(1)=0$. Furthermore,  $\var{f^{sal}}=\sum_{u\in\mc{S}} \abs{s_u}\leq \var{f} $.

\begin{prop}[\textbf{Key facts about the invariant densities}]\label{P:densities}
There exists $\ep_0 > 0 $ such that:
\begin{enumerate}

\item[(i)] \emph{Uniform bound on the variations of the invariant densities.}\\
  \[
        \sup_{0<\ep<\ep_0} \var{\phi_\ep}<+\infty.
  \]
  Also, $\var{\phi_\lt},\var{\phi_\rt}<+\infty $.

\item[(ii)] \emph{Uniform bound on the Lipschitz constant of the regular parts}.\\
  For $0<\ep<\ep_0$, each of the $\phi_\ep^{reg} $ is Lipschitz continuous with
  constant $\lip{\phi_\ep^{reg}}$, and
  \[
        \sup_{0<\ep<\ep_0} \lip{\phi_\ep^{reg}}<+\infty.
  \]
  Also, $\phi_\lt^{reg}$ and $\phi_\rt^{reg}$ are Lipschitz.

\item[(iii)] \emph{Approximate continuity near the infinitesimal holes.} \\
    For $*\in\{\lt,\rt\}$, for each $\eta>0$, there exists $\delta>0$ such that for all $0<\ep<\ep_0$,
    \[
        \varsub{[h_*-\delta,h_*+\delta]}{\phi_\ep^{sal}}: =\text{the variation of } \phi_\ep^{sal} \text{ over } [h_*-\delta,h_*+\delta]<\eta.
    \]
    Also, $\phi_* $ is continuous at $h_* $.

\end{enumerate}
\end{prop}

The proof of Proposition~\ref{P:densities} is technical, and so we defer it until \S\ref{S:densitiesproof}.

\subsection{Proofs}\label{S:main_proof}

We recall that for any $C_1,C_2>0 $, $\{f\in BV:\abs{f}_{L^1}\leq C_1, \var{f}\leq C_2\}$ is pre-compact in $L^ 1 $. This fact will be used repeatedly in what follows.

\subsubsection*{Proof of Theorem~\ref{thm1}}

Using (i) of Proposition~\ref{P:densities}, we are able to choose a sequence of values $\ep'$ converging to $0$ such that  $\phi_{\ep'}$ converges in $L^1$ to some function, which we denote by $\phi_0 $.
Using the fact that $\phi_\ep$ is a fixed point of the Perron-Frobenius operator $\L_\ep$ associated to $T_\ep$ (see \S\ref{S:reviewLY} for the definition), one can verify that $\phi_0 $ is an invariant density for $T_0 $, and so there exists $\al $ such that $\phi_0=\al \phi_\lt+(1-\al)\phi_\rt$.  We will verify that necessarily $\al/(1-\al)=l.h.r. $.
From this it follows that there is exactly one limit point of $\phi_\ep$ as $\ep\rightarrow0$, and Theorem~\ref{thm1} follows.

Now for $\ep' > 0$, $\mu_{\ep'}(H_{\lt,\ep'})=\mu_{\ep'}(H_{\rt,\ep'})$, because $\phi_{\ep'}=d\mu_{\ep'}/dx$ is an invariant density for $T_{\ep'} $.  We will show that as $\ep'\rightarrow 0$,
\begin{eqnarray}
  \mu_{\ep'}(H_{\lt,\ep'}) &=& \al\mu_\lt(H_{\lt,\ep'})+o(1)\cdot\mu_\lt(H_{\lt,\ep'}),\label{eq:Hltmeasure} \\
  \mu_{\ep'}(H_{\rt,\ep'}) &=& (1-\al)\mu_\rt(H_{\rt,\ep'})+o(1)\cdot\mu_\rt(H_{\rt,\ep'}),\label{eq:Hrtmeasure}
\end{eqnarray}
from which the equation $\al/(1-\al)=l.h.r. $ and hence Theorem~\ref{thm1}  follows immediately.

We prove only Equation~\eqref{eq:Hltmeasure}, since the proof of Equation~\eqref{eq:Hrtmeasure} is analogous.  Write
\[
\begin {split}
    \mu_{\ep'}(H_{\lt,\ep'})&=\int_{H_{\lt,\ep'}}\phi_{\ep'}\,dx
    =\al\int_{H_{\lt,\ep'}}\phi_{\lt}\,dx+\int_{H_{\lt,\ep'}}(\phi_{\ep'}-\al\phi_\lt)\,dx
    \\
    &=\al\mu_\lt(H_{\lt,\ep'})+
    O\left(\sup_{x\in H_{\lt,\ep'}}\abs{\phi_{\ep'}(x)-\al\phi_\lt (x)}\right)\cdot\text{Leb}(H_{\lt,\ep'}) .
\end {split}
\]
But as $\ep'\rightarrow 0 $, $H_{\lt,\ep'}\rightarrow h_\lt $ in the Hausdorff metric, and then $\mu_\lt(H_{\lt,\ep'})/\text{Leb}(H_{\lt,\ep'})\rightarrow \phi_\lt(h_\lt)>0$, because $\phi_\lt$ is continuous at $h_\lt$.  Thus our proof is completed by the following:
\begin{lemma}
As $\ep'\rightarrow 0 $,
\[
    \sup_{x\in H_{\lt,\ep'}}\abs{\phi_{\ep'}(x)-\al\phi_\lt (x)} \rightarrow 0.
\]

\end{lemma}

Although this uniform convergence might at first seem surprising, Proposition~\ref{P:densities} (ii) and (iii) essentially say that near $h_\lt $, $\{\phi_{\ep'}\} $ behaves like a family of equicontinuous functions.

\begin{proof}
We proceed by contradiction.  Suppose that there exists $C>0 $ and a subsequence $\ep''\rightarrow 0$  of the $\ep'$ values such as that for each $\ep'' $, there is a point $x_{\ep''} \in H_{\lt,\ep''}$ with $\abs{\phi_{\ep''}(x_{\ep''})-\al\phi_\lt (x_{\ep''})}>C$.  Necessarily, $x_{\ep''}\rightarrow h_\lt $ as $\ep''\rightarrow 0$.

 We restrict all functions of interest to the left subinterval $I_\lt $.  Set $\ga_{\ep''} := \phi_{\ep''}^{reg} -\al\phi_\lt^{reg}$ and $\omega_{\ep''} := \phi_{\ep''}^{sal} -\al\phi_\lt^{sal}$, so that $\phi_{\ep''}-\al\phi_\lt= \ga_{\ep''}+\omega_{\ep''}$.  Using (ii) of Proposition~\ref{P:densities}, let $L$ be such that for all sufficiently small $\ep'' $, $\lip{\ga_{\ep''}}<L $.  Next, we use (iii) with $\eta=C/5$ and make sure to choose the corresponding $\del<C/ (5L)$ small enough so that $\varsub{[h_\lt-\del,h_\lt+\del]}{\al\phi_{\lt}^{sal}}<C/5$ as well.  Thus $\varsub{[h_\lt-\del,h_\lt+\del]}{\omega_{\ep''}}<2C/5$.  Then if $x\in [h_\lt-\del,h_\lt+\del] $, and $\ep'' $ is sufficiently small, $x_{\ep''}\in [h_\lt-\del,h_\lt+\del] $ and
\begin{align*}
    &\abs{ \ga_{\ep''}(x)+\omega_{\ep''}(x) }\\
    &\geq  \abs{\ga_{\ep''}(x_{\ep''})+\omega_{\ep''}(x_{\ep''})}
    -\abs{\ga_{\ep''}(x)+\omega_{\ep''}(x)-\ga_{\ep''}(x_{\ep''})-\omega_{\ep''}(x_{\ep''})}
    \\
    &\geq  C - [L\cdot 2\del +2C/5]
    \geq C/5.
\end {align*}
But this contradicts that $\ga_{\ep''}+\omega_{\ep''}=\phi_{\ep''}-\al\phi_\lt\overset{L^1}{\longrightarrow}0$.
\end{proof}

\subsubsection*{Proof of Theorem~\ref{thm2}}

First, we observe that the results of \cite{KellerLiverani99}
guarantee that for small $\ep >0 $,  $\rho_\ep<1 $ is a simple eigenvalue of multiplicity 1.  Hence there are exactly two real-valued eigenfunctions, $\pm\psi_\ep$, satisfying $\L_\ep\psi_\ep=\rho_\ep\psi_\ep$ and $\abs{\psi_\ep}_{L^1}=1$.  But for such functions,
$
  \int \psi_\ep\, dx= \int \L_{\ep} \psi_\ep \,dx= \rho_\ep \int \psi_\ep \, dx,
$
so $ \int \psi_\ep \, dx= 0$.  We have the following uniform bound on their variations, whose proof we defer until \S\ref{S:densitiesproof}.
\begin{lemma}[\textbf{Uniform bound on the variations of the $\psi_\ep$}]\label{L_psibound}
There exists $\ep_1>0 $ such that
$
        \sup_{0<\ep<\ep_1} \var{\psi_\ep}<+\infty.
$
\end{lemma}

Let $\psi_0$ be any limit point in $L^1$ of $\psi_\ep$ as $\ep\rightarrow 0$.  Then, since $\rho_\ep\rightarrow 1$, it follows that $\psi_0$ is invariant under $\L_0$, and is thus a linear combination of $\phi_\lt$ and $\phi_\rt$.  Since $\abs{\psi_0}_{L^1}=1$ and $ \int \psi_0 \, dx= 0$, necessarily $\psi_0=\pm\frac{1}{2}\phi_\lt \mp\frac{1}{2}\phi_\rt.$  Hence we can uniquely specify $\psi_\ep $ by the condition $\int_{I_\lt} \psi_\ep \,dx>0 $, and Theorem~\ref{thm2} follows.

\section{Proofs of the properties of the densities}

In order to prepare for the proofs of Proposition~\ref{P:densities} and Lemma~\ref{L_psibound}, it will be convenient to first show how to derive such properties for an invariant density of a single, fixed piecewise expanding map.  We do this in \S\ref{S:reviewLY}. Then, in \S\ref{S:densitiesproof}, we prove Proposition~\ref{P:densities} and Lemma~\ref{L_psibound} by showing how such estimates can be made uniformly for the family of maps $T_\ep $, $\ep\geq 0$.

Before beginning, we remark that if $f\in BV $, then for each $x$, $\abs{f}_\infty\leq  \abs{f(x)}+\var{\abs{f}}\leq  \abs{f(x)}+\var{f}$.  Integrating, we find that $\abs{f}_\infty\leq \abs{f}_{L^ 1}+\var{f}$.  We will use this fact repeatedly below.

\subsection{Properties of an invariant density for a single piecewise expanding map} \label{S:reviewLY}

 Let $T:I\circlearrowleft$ be a piecewise $C^2$ uniformly expanding map, with  $\C=\{ 0=c_0 < c_1<  \cdots < c_d=1\}$ as a critical set.    Let $\L$ be the associated Perron-Frobenius operator, i.e., the transfer operator acting on densities.  We begin by briefly reviewing a method for finding an invariant density of $T$. Such a method was introduced in~\cite{LasotaYorke}; see Chapter 3 in \cite{BaladiBook} for a more modern exposition.  Let $\lam_T = \inf_{x\in I\setminus \C}\abs{T'(x)}>1$ be the minimum expansion and $D_T = \sup_{x\notin\C}\abs{T''(x)}/\abs{T'(x)}$ be the distortion of $T$. Then if $f\in BV $, $x\notin T\C$,
 \begin{equation}\label{E:transfer}
    \L f(x)=\sum_{i=1}^d f(\xi_i(x))\abs{\xi_i'(x)}1_{J_i}(x),
 \end{equation}
 where $J_i=T\vert_{[c_{i-1},c_i]}([c_{i-1},c_i])$ and $\xi_i=(T\vert_{[c_{i-1},c_i]})^{-1}:J_i\rightarrow [c_{i-1},c_i]$.  One can show that there exists constants $\beta\in (0,1) $  and $\cly $ such that for each $n\geq 1$ and $f\in BV $, the following Lasota-Yorke inequality holds:
 \begin{equation}\label{E:LYinequality}
    \var{\L^n f}\leq \cly\beta^{n}\var{f}+\cly\abs{f}_{L^ 1}.
 \end{equation}
In fact, $\beta $ can be chosen as any number greater than $\lam_T^{-1}$, although we will not use this fact.  Set $F_n=\frac{1}{n}\sum_{k=0}^{n-1}\L ^k 1$.  Then $F_n\overset{L^1}{\longrightarrow}\phi$, where $\phi\in BV$ is the density of an \acim~for $T$.  Using Helly's Theorem,  one has that $ \var {\phi}\leq\cly $.

We wish to characterize the properties of the regular and singular terms in the decomposition $\phi=\phi^{reg} +\phi^{sal} $.  First, let us define a hierarchy on the set of points in the postcritical orbits $\mc{S}=\cup_{k\geq 1}T^k\C $ by $\#(u):=\inf\{k\geq 1:u\in T^k\C\}$.
The following characterization is motivated by the discussion of the invariant densities for unimodal expanding maps found in~\cite{Baladi} and \cite{BaladiSmania}.  In particular, in~\cite[\S3.3]{BaladiSmania} a norm is introduced on the sequence of jumps of $\phi$ along the postcritical orbit with weights that grow exponentially in $\#(u) $.
 \begin{lemma}\label{L:expdecay}
 Given the hypotheses above,
\begin{itemize}
  \item[(a)]
$\phi^{reg} $ is Lipschitz continuous.  Furthermore, there exists a constant $\cdis =\cdis (\lam_T,D_T)$ such that $\lip{\phi^{reg}} \leq \cdis (1+\cly)$.  $\cdis (\lam_T,D_T)$ can be defined so that it depends continuously on $\lam_T>1 $, $D_T\geq 0 $.
\item[(b)]
  The discontinuity set of $\phi$ is a subset of $\mc{S}=\cup_{k\geq 1}T^k\C $.  If we write $\phi^{sal} =\sum_{u\in\mc{S}} s_u H_{u} $, then
  for each $m\geq 0$,
    $\sum_{\{u\in \mc{S}:\# (u) > m\}}\, \abs{s_u} \leq \lam_T^{-m}\, \cly$.
\end{itemize}
 \end{lemma}

 \begin{proof}

  We begin by noting from Equation~\eqref{E:transfer} that $F_n $ is smooth except possibly at points in $\cup_1^{n-1} T^k\C$.   Write $F_{n}^{sal}=\sum_{u\in\mc{S}} s_{u,n}H_u$.  Then we can show that $\abs{s_{u,n}}$ decays uniformly exponentially fast in $\#(u)$, i.e.
  \begin{sublemma}\label{L:expdecay0}
 For each $m,n\geq 0$,
    $\sum_{\{u\in \mc{S}:\,\# (u) > m\}}\, \abs{s_{u,n}} \leq \lam_T^{-m}\,\cly.$
 \end{sublemma}

 \begin{proof}[Proof of the Sublemma]
 If $m\geq n$, $\sum_{\# (u) > m}\abs{s_{u,n}}= 0 $, and $\sum_{\# (u) > 0}\abs{s_{u,n}}= \var{F_n^{sal}}\leq \var{F_n}\leq\cly$.  Since $F_n =\frac{n-1}{n} \L F_{n-1} +\frac{1}{n}$, if $\#(u)>1 $ we see from Equation~\eqref{E:transfer} with $f=F_{n-1}$ that
 $\abs{s_{u,n}}\leq \frac{n-1}{n}\lam_T^{-1}\sum_{\{v\in \mc{S}:\, Tv=u\}} \, \abs{s_{v,n-1}} $.  Thus if $0<m < n $,
 \[
    \begin{split}
    \sum_{\# (u) > m}\abs{s_{u,n}}
    &\leq
    \sum_{\# (u) > m} \frac{n-1}{n}\lam_T^{-1}\sum_{\{v\in \mc{S}:\, Tv=u\}} \, \abs{s_{v,n-1}}
    \\ \leq
    &\frac{n-1}{n}\lam_T^{-1}\sum_{\# (u) > m-1}\abs{s_{u,n-1}}
    \leq \cdots \\
    \leq
    &\frac{n-m}{n}\lam_T^{-m}\sum_{\# (u) > 0}\abs{s_{u,n-m}}
    \leq
    \lam_T^{-m}\cly.
    \end{split}
 \]
In the inequalities above, we use the fact that if $\# (u) >1$, then $T^{-1}(u)$ does not contain any critical points.
\end{proof}

Using a diagonalization argument, we may find a subsequence $n_j$ such that for each $u$, $s_{u,n_j}$ converges as $n_j\rightarrow\infty$ to some number, which we write as $\hat{s}_u$.  In particular, for each $m$, $\sum_{\# (u) > m}\abs{\hat{s}_u}\leq\lam_T^{-m}\cly$, and $F_{n_j}^{sal} \overset{L^1}{\longrightarrow} F^{sal}$, where we define $F^{sal}=\sum_{u\in\mc{S}} \hat{s}_u H_u$.
Furthermore, a standard distortion estimate (see for example the proof of Proposition~3.3 in~\cite{Baladi}) shows that there exists a constant $\cdis $ such that for each $n\geq 0$, $\lip{(\L ^n 1)^{reg}}\leq \cdis\abs{\L ^n 1}_\infty \leq \cdis (1+\var{\L ^n 1})$.  Here, $\cdis $ depends only on the minimum expansion and on the distortion of $T $.  In particular,  $\sup_{n\geq 1}\lip{F_n^{reg}}\leq \cdis (1+\cly)$.   By the Arzel\`{a}-Ascoli Theorem, we may find a continuous function $F^{reg} $ such that some subsequence of $\{ F_{n_j}^{reg}\}$ converges in $L^{\infty}$ to $F^{reg}$.

 By the uniqueness of the decomposition $\phi=\phi^{reg} + \phi^{sal} $, we conclude that $\phi^{reg}=F^{reg} $ and $\phi^{sal} =F ^{sal} $.  Lemma~\ref{L:expdecay} follows.

\end{proof}

\subsection{Proofs of Proposition~\ref{P:densities} and Lemma~\ref{L_psibound}}\label{S:densitiesproof}

We prove only the claims about $\phi_\ep $ for $\ep>0 $, and leave the claims about $\phi_\lt, \phi_\rt $ to the reader.

Let $\L_\ep$ be the Perron-Frobenius operator \eqref{E:transfer} associated to $T_\ep$. The first key step is to prove that the $\L_\ep $ with $\ep $ sufficiently small satisfy Lasota-Yorke inequalities with uniform constants.  Let $\lam_\ep$ and $D_\ep $ be the minimum expansion and distortion of $T_\ep $, respectively.  Then as $\ep\rightarrow 0$, $\lam_\ep\rightarrow \lam_0$ and $D_\ep\rightarrow D_0 $.  Furthermore, $T_\ep $ is a piecewise $C^2$ uniformly expanding map that is a small $C^2$ perturbation of $T_0$,
and the two critical sets $\C_\ep$, $\C_0 $ are $\ep-$close together.  This is not sufficient to guarantee uniform Lasota-Yorke inequalities, see for example \cite[\S 6]{Keller82} or \cite{Blank}. However, such uniform  inequalities do follow with the additional assumption (I4), which guarantees that either (a) we have $\lam_0>2 $ or (b) $T_0$ has no periodic critical points, except possibly the points in $\partial I $ as fixed points.  We assume the former case in our presentation here, and comment on the latter case at the end of this section.

Fix $\lam\in(2,\lam_0) $.  The original proof from \cite{LasotaYorke} shows that if $f\in BV $ is real-valued,
 \begin{equation*}
    \var{\L_\ep f} \leq  (2\lam_\ep^{-1})\var{f}+C_\ep\abs{f}_{L^ 1},
 \end{equation*}
where
 \begin{equation}\label{E:LYconst}
    C_\ep =  D_\ep/\lam_\ep+2 \max_i\abs{c_{i+1,\ep}-c_{i,\ep}}^{-1}.
 \end{equation}
(Compare also~\cite{Liverani}[\S2].)   Iterating, we find that for sufficiently small $\ep$, for all such $f$ and $n\geq 1 $,
 \begin{equation}\label{E:unifLYinequality2}
    \var{\L_\ep^n f} \leq  \beta^{n}\var{f}+\cly\abs{f}_{L^ 1},
 \end{equation}
with $\beta=2\lam^{-1} $ and $\cly=2C_0/(1-2\lam^{-1}) $.
Similar estimates can be made for complex-valued $f $ by applying \eqref{E:unifLYinequality2} to the real and imaginary parts separately.  Since each $T_\ep $ has a unique \acim, we know from our discussion in \S\ref{S:reviewLY} that for sufficiently small $\ep>0$, $\frac{1}{n}\sum_{k=0}^{n-1}\L_\ep ^k 1 \overset{L^1}{\longrightarrow}\phi_\ep $ as $n\rightarrow\infty$.  It follows from Lemma~\ref{L:expdecay} that $\var{\phi_\ep}$ and $\lip{\phi_\ep^{reg}}  $ are uniformly bounded.

Next, we prove (iii).  Given $\eta>0$, choose $n$ large enough that $\lam^{-n}\cly<\eta$.  Using (I2), we can choose $\del>0$ so small that for $0<k\leq n $, $(T_0^k\C_0)\cap [h_*-2\del,h_*+2\del]=\emptyset$.  It follows that for $\ep$ sufficiently small, $(T_\ep^k\C_\ep)\cap [h_*-\del,h_*+\del]=\emptyset$ as well.  Using part (b) of Lemma~\ref{L:expdecay} with $m=n$, we then see that $\varsub{[h_*-\delta,h_*+\delta]}{\phi_\ep^{sal}} <\eta.$

Finally, to prove Lemma~\ref{L_psibound}, we use Equation~\eqref{E:unifLYinequality2} with $f=\psi_\ep $, $n $ chosen so large that $\beta^ {n}<1/2$, and $\ep$ chosen so small that $\rho_\ep^n >3/4 $.    It follows that $\var{\psi_\ep}\leq\cly/(\rho_\ep^n - \beta^ {n})\leq 4 \cly$.

\subsubsection*{Modifications when the minimum expansion is not bigger than two}
If, in assumption (I4), the minimum expansion of $T_0$ is $\lam_0\leq 2 $, one derives Lasota-Yorke estimates for $\L_0 $ by first fixing $N $ large enough so that $\lam_0 ^N >2 $.  Then the arguments from~\cite{LasotaYorke} used above will yield a Lasota-Yorke estimate for $\L_0^N $, and this can be interpolated to give similar estimates for $\L_0 $.  One can try to obtain uniform estimates for $\L_\ep$, but the arguments used above will only work if the critical points for $T_\ep^N $ are in a one-to-one correspondence with and very close to those of $T_0^N$,  compare Equation~\eqref{E:LYconst}, as would be the case if $\C_0\cap(\cup_{k=1}^{N-1}T^k\C_0) =\emptyset$.

\comment{
\begin{figure}[htbp]
\begin{center}
\resizebox{12cm}{!}{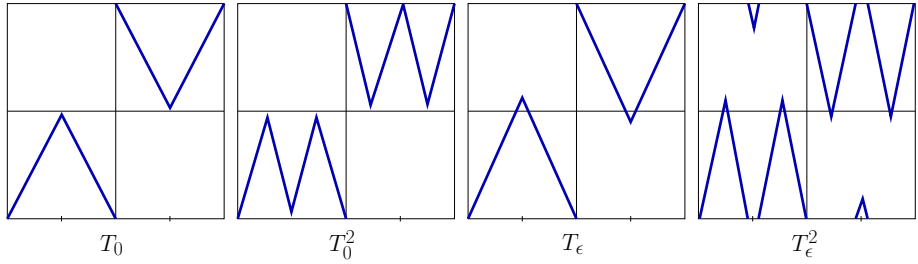}
\caption{Creation of small intervals of differentiability.}
\label{figure:smallintervals}
\end{center}
\end{figure}
}

Unfortunately, this will never be the case in our setting, at least when $\bp\in\C_0 $.  This is because the infinitesimal holes in $H_0 $ are necessarily critical points, and they are mapped to $\bp $ by $T_0 $.  Because at least some of the infinitesimal holes must be mapped across the boundary point when $\ep>0$, this means that necessarily $T_\ep^2$ will have more critical points than $T_0^2 $, and these additional critical points will create very short intervals on which $T_\ep^2$ is smooth; see Figure~\ref{figure:smallintervals}. However, this problem can be dealt with using assumption (I4b).  Specifically, in~\cite{BaladiYoung} it is shown that because of the restriction on the periodic critical points, the growth in the number of the very short intervals on which $T_\ep^n$ is smooth as $n $ increases can be controlled, and that uniform Lasota-Yorke estimates can still be made.  Precisely, there exists $\ep_0>0 $ and constants $\beta\in(0,1) $, $\cly$ such that for each $\ep\in [0,\ep_0] $, $n\geq 1$ and $f\in BV $,
 \begin{equation*}\label{E:unifLYinequality}
    \var{\L_\ep^n f} \leq  \cly\beta^{n}\var{f}+\cly\abs{f}_{L^ 1}.
 \end{equation*}
The proof of this is essentially identical to the proof of Lemma~3.2 in \cite{BaladiBook} (see also her Remark~3.4, and compare the proof of Lemma 8 in~\cite{BaladiYoung}), and so we omit it.
The rest of the proofs of Proposition~\ref{P:densities} and Lemma~\ref{L_psibound} proceed as above.

\subsubsection*{Acknowledgments}
The authors would like to especially acknowledge D.~Dolgopyat for suggesting this project and for helpful conversations.  They also thank M.~Demers for helpful conversations.

\bibliography{imms}
\bibliographystyle{alpha}
\end{document}